\documentclass[final]{AIMS}

\usepackage{amsmath}
  \usepackage{paralist}
  \usepackage{graphics} 
  \usepackage{epsfig} 
\usepackage{graphicx}  \usepackage{epstopdf}
 \usepackage[colorlinks=true]{hyperref}
\hypersetup{urlcolor=blue, citecolor=red}

\usepackage{lipsum}
\usepackage{amsfonts}
\usepackage[color]{showkeys}
\definecolor{refkey}{gray}{.75}

\def\currentvolume{X}
 \def\currentissue{X}
  \def\currentyear{200X}
   \def\currentmonth{XX}
    \def\ppages{X--XX}
     \def\DOI{10.3934/xx.xx.xx.xx}

\newtheorem{theorem}{Theorem}[section]

\newtheorem{lemma}[theorem]{Lemma}

\theoremstyle{definition}

\def\p{\partial}
\def\tilde{\widetilde}

\def\bs{\boldsymbol}

\def\xb{\bs{x}}
\def\ab{\bs{a}}
\def\nb{\mathbf{n}}
\def\ub{\bs{u}}
\def\vb{\bs{v}}

\def\qb{\bs{q}}
\def\psib{\bs{\psi}}
\def\lb{\bs{l}}
\def\phib{\bs{\varphi}}
\def\fb{\bs{f}}
\def\hb{\bs{h}}
\def\gammab{\bs{\gamma}}

\def\C{\mathcal{C}}
\def\D{\mathcal{D}}
\def\M{\mathcal{M}}
\def\Mbar{\overline{\mathcal{M}}}
\def\lpm{{L^p(\M)}}
\def\linfm{L^\infty{\M}}
\def\l2m{L^2(\M)}
\def\w1pm{W^{1,p}(\M)}
\def\h10m{H^1_0(\M)}

\ifpdf
  \DeclareGraphicsExtensions{.eps,.pdf,.png,.jpg}
\else
  \DeclareGraphicsExtensions{.eps}
\fi

\title[Multi-layer QG] 
      {On the well-posedness of the inviscid multi-layer
        quasi-geostrophic equations}

\author[Qingshan Chen]{}


 \email{qsc@clemson.edu}

\thanks{The first author is supported by Simons Foundation (xx-xxxx)}

\thanks{$^*$ Corresponding author: Qingshan Chen}

\usepackage{amsopn}

\begin{document}
\maketitle

\centerline{\scshape Qingshan Chen $^*$}
\medskip
{\footnotesize
 \centerline{Department of Mathematical Sciences}
   \centerline{Clemson University}
   \centerline{ Clemson, SC 29631, USA}
} 

\bigskip

 \centerline{(Communicated by the associate editor name)}

\begin{abstract}
The inviscid multi-layer quasi-geostrophic equations are considered
over an arbitrary bounded domain. The no-flux but non-homogeneous
boundary conditions are imposed to accommodate the free fluctuations
of the top and layer interfaces. Using the barotropic and baroclinic
modes in the vertical direction,  the  elliptic system 
governing the streamfunctions and the potential vorticity is
decomposed into a sequence of scalar elliptic boundary value problems,
where the regularity theories from the two-dimensional case can be
applied. With the initial potential vorticity being essentially
bounded, the multi-layer quasi-equations are then shown to be globally
well-posed, and the initial and boundary conditions are satisfied in
the classical sense. 
\end{abstract}

\section{Introduction}
At the mid-to-high latitudes, where the Coriolis parameter is away
from zero, large-scale geophysical flows, namely the ocean and
atmosphere, evolve around the so-called geostrophic balance, where the
Coriolis force approximately counteract the horizontal pressure
gradient. To the leading order, the flow is governed by the
quasi-geostrophic (QG) equations. The QG equations take the form of a
transport equation,
\begin{equation}
  \label{eq:intro1}
  \dfrac{\p}{\p t} q + \ub\cdot\nabla q = F,
\end{equation}
where $q$ represents the QG potential vorticity (PV), $\ub$ the
horizontal velocity field, and $F$ on the right-hand side is a place
holder for other terms in the dynamics, such as the external forcing,
the diffusion, etc. In the QG, the velocity $\ub$ can be derived from
the QGPV $q$, and therefore, the QG, together with the suitable
initial and boundary conditions, can be viewed as a closed system
about a single quantity, the QGPV $q$. This simple and yet
sophisticated model provides a unified framework for studying both the
ocean and atmosphere (\cite{Pedlosky1987-gk, Vallis2017-jj,
  Lorenz2006-zx}).  

Depending on the assumption on the vertical density profile, the QG
equation(s) can take the form of a single scalar two-dimensional
equation (the barotropic case with a uniform density profile), a
system of two-dimensional scalar equations (the multi-layer case with
a non-uniform discrete density profile), or a three-dimensional scalar
equation (the 3D case with a non-uniform but continuous density
profile). The QG equations form a hierarchy of models, with increasing
complexity, for the large-scale geophysical flows. As a reference
regarding the complexity, the barotropic QG is on the same level as
the two-dimensional incompressible Euler equations; it is only a step
forward from the latter with the inclusion of a free surface on the
top. While the three-dimensional QG equation is posed on a
three-dimensional spatial domain, the velocity vector at every point
is horizontal, and therefore two-dimensional. Thus, the
three-dimensional QG is simpler than the three-dimensional
incompressible Euler equations, and hopefully, more amenable as well. 

Several authors have
studied the three-dimensional QG equation under idealized settings, in
the unbounded half space, or a rectangular box. An early work
is by Dutton (\cite{Dutton1974-oo, Dutton1976-kh}), who considered the
three-dimensional QG model in a rectangular box with periodic boundary
conditions on the sides, and homogeneous Neumann boundary conditions
on the top and bottom. The uniqueness of a classical solution, if it
exists, and the global existence of a generalized solution were
established. Bourgeois and Beale (\cite{Bourgeois1994-tv}) studied the
equation in a similar setting, and the existence of a global strong
solution was proved. Desjardins and Greneier
(\cite{Desjardins1998-hb}) also considered the equation in a similar
setting, but included in their model the Ekman pumping effect which
effectively add diffusion to the flow. The existence of a global weak
solution is given. Puel and Vasseur (\cite{Puel2015-mw}) considered
the inviscid QG in the upper half space, with the non-penetration
boundary condition at the bottom of the fluid. The global existence of
a weak solution was proven.  In these works, the issue of uniqueness
of the solutions was left open. In a recent work, Novack and Vasseur
(\cite{Novack2016-op}) considered the three-dimensional QG in the same
spatial setting as in \cite{Puel2015-mw}, but with an added diffusion
term in the boundary at $z=0$ due to the Ekman pumping effect. The
existence and uniqueness of a global strong solution is proven. 
Novack (\cite{Novack2017-dc}) studies the existence of a weak solution
to the inviscid 3D QG equation, with initial data in the Lebesgue
spaces. The present work focuses on simpler models, but on more
general settings, namely purely inviscid models on arbitrary bounded
domains with physically relevant boundary conditions. 

Th well-posedness of the barotropic QG equation with a free top
surface is the subject of a previous work (\cite{Chen2017-fh}). The
goal of the 
present work is to address the issue of well-posedness for the
multi-layer QG equations. 
Within the multi-layer QG
equations, each layer behaves like a barotropic QG, and the layers
interact with each other through pressure. Because of these
interactions, the well-posedness of the barotropic QG does not
directly transfer over to the multi-layer case. The layer interactions
make the problem more interesting and more challenging at the same
time.

A major challenge in the previous work (\cite{Chen2017-fh}) is the
non-homogeneous boundary conditions on the streamfunction, which is
imposed to accommodate the free fluctuations of the top
surface. There, the challenge is dealt with by the superposition rule
and an estimate on the constant non-zero value of the
streamfunction. In this work, not only is the top surface left free,
but also are the interior interfaces between layers. It turns out that the
interior interfaces behave like the top surface, and can be treated
as such. Therefore, the same type of boundary conditions are imposed
on the interior interfaces, and they are treated in exactly the same
way as in \cite{Chen2017-fh}. 

For both the ocean and the atmosphere, the density of the fluid is
non-uniform, which is the basis for the multi-layer or
three-dimensional models. Not only so, the rate at which density
varies against the height (or depth for the ocean) is also
non-uniform. For example, in the ocean, the density of the water
increases rapidly downward for the first couple of hundred meters, and
then stay almost flat for the next thousands of meters
(\cite{Pedlosky1987-gk, Talley2011-ak}). Because
of this non-uniform changing rate, the vertical interaction between
layers takes the form of a second-order derivative with a non-uniform
coefficient, in the continuous case,
\begin{equation*}
  \dfrac{\p}{\p z}\left(\dfrac{1}{S(z)}\dfrac{\p\psi}{\p z}\right),
\end{equation*}
where $\psi$ stands for the streamfunction, and $S(z)>0$ is determined
by the vertical density profile. Under 
the usual homogeneous Neumann boundary conditions for $\psi$, this
operator is self-adjoint. In the multi-layer case, the non-uniformity
in the changing rate gives rise to a non-symmetric matrix with
non-positive eigenvalues, and the layer interactions are represented
by a matrix-vector product,
\begin{equation*}
  L\psib,
\end{equation*}
where $\psib$ is a vector-valued function representing the
streamfunction across the layers, and $L$ is non-symmetric coefficient
matrix. 
Because of its non-symmetry, even though $L$ has only non-negative
eigenvalues, the inner product
\begin{equation*}
  (L\psib,\,\psib),
\end{equation*}
which appears in the analysis of the elliptic boundary value problem
governing the streamfunction $\psib$ and the QGPV $\qb$, is not
negative definite. This lack of definiteness is not fatal for the
analysis, as it can be remedied by a decomposition in the eigenmodes
of the associated Sturm-Liouville problem in the vertical
direction. The major hurdle, as it turns out, is related to another
inner product involving a time derivative. Due to the non-symmetry of
$L$, the inner product
\begin{equation*}
  \left(\dfrac{\p}{\p t}(L\psib),\,\psib\right)
\end{equation*}
is no longer exactly integrable in time. To circumvent this
difficulty imposed by the physical reality, we assume, in this study,
that the density profile is linear with respect to the height, and the
coefficient matrix $L$ in the multi-layer case is actually
symmetric. Of course, as pointed out above, this assumption runs
against the physical reality. But this assumption does not
significantly compromise the mathematical generality of the problem,
because layer 
interactions are still included in the model. We also note that 
the corresponding differential operator in the continuous
three-dimensional case 
is actually self-adjoint, which is the analogue of the symmetry
of the discrete operator, and thus the current work can still serve as
a stepping stone to the three-dimensional problem. 

The rest of the paper is organized as follows. Section \ref{s:ibvp}
presents the initial-boundary value problem for the multi-layer QGs in
its complete form. Section \ref{s:bvp} deals with an elliptic boundary
value problem associated with the multi-layer QG. In Section
\ref{s:weak}, a weak formulation and some {\itshape a priori} results
are obtained. Section \ref{s:unique} is devoted to the uniqueness of
the weak solution, and Section \ref{s:exist} to the existence of this
solution. The paper concludes in Section \ref{s:conclude}.

\section{The initial and boundary conditions}\label{s:ibvp} 

We consider a 3-layer QG system,
\begin{equation}
  \label{eq:1}
  \left\{
    \begin{aligned}
      &\left(\dfrac{\p}{\p t} + \ub_1\cdot\nabla\right)\left(\zeta_1 +
        \beta y + F_1^2 (-\psi_1 + \psi_2)\right) = f_1,\\
      &\left(\dfrac{\p}{\p t} + \ub_2\cdot\nabla\right)\left(\zeta_2 +
        \beta y + F_2^2 (\psi_1 -2 \psi_2 + \psi_3)\right) = f_2,\\
      &\left(\dfrac{\p}{\p t} + \ub_3\cdot\nabla\right)\left(\zeta_3 +
        \beta y + F_3^2 (\psi_2 - \psi_3)\right) = f_3.
    \end{aligned}\right.
\end{equation}
In the above, for each $i = 1,\,2,\,3$,
\begin{align*}
  &\psi_i  & &\textrm{Pressure perturbation},\\
  &\ub_i = \nabla^\perp \psi_i,  & &\textrm{Horizontal velocity,}\\
  &\zeta_i = \nabla\times\ub_i, & &\textrm{Relative vorticity,}\\
  & F_i \equiv \dfrac{L}{\sqrt{g'D_i}/f_0}, & &\textrm{Froude number.}
\end{align*}
In the specification for the Froude number, $L$ represents the typical
horizontal length scale of the flow, $D_i$ the average layer depth,
and $g'$ is the reduced gravity within the flow.  

This study focuses on the effect of the nonlinearity within each
layer, as well as the interaction between the layers. For this reason,
the diffusion terms have been omitted. 

In reality, the Froude number $F_i = O(1)$. We therefore take $F_i =
1$ in the equations \eqref{eq:1}. This choice has the added benefit
that the coefficient matrix for the zeroth order terms is now
symmetric, the significance of which has been discussed in the
Introduction. The beta terms are mathematically insignificant, and
therefore they will be neglected from now on. Thus we consider the
following model, 
\begin{equation}
  \label{eq:2}
  \left\{
    \begin{aligned}
      &\left(\dfrac{\p}{\p t} + \ub_1\cdot\nabla\right)\left(\zeta_1 +
        (-\psi_1 + \psi_2)\right) = f_1,\\
      &\left(\dfrac{\p}{\p t} + \ub_2\cdot\nabla\right)\left(\zeta_2 +
        (\psi_1 -2 \psi_2 + \psi_3)\right) = f_2,\\
      &\left(\dfrac{\p}{\p t} + \ub_3\cdot\nabla\right)\left(\zeta_3 +
        (\psi_2 - \psi_3)\right) = f_3.
    \end{aligned}\right.
\end{equation}
The variables $\psi_i$, $\ub_i$, and $\zeta_i$ are defined as before. 

For this inviscid system, the no-flux boundary conditions are imposed
on the velocity field, and in terms of the streamfunctions, these
conditions can be written as 
\begin{equation}
  \label{eq:3}
  \psi_i = \textrm{constant} \qquad \textrm{for each } 1\le i\le
  3\textrm{ on } \p\Omega.
\end{equation}
In order to uniquely determine the streamfunctions, a mass conservation
constraint is imposed on each layer, 
\begin{equation}
  \label{eq:4}
  \int_\Omega \psi_i(x,t) dx = 0.
\end{equation}
Finally, the initial conditions are imposed on the streamfunctions as
well, 
\begin{equation}
  \label{eq:5}
  \psi_i(x,0) = \psib^0_i(x),\qquad\textrm{for each } 1\leq i\leq
  3\textrm{ and } \forall x\in\Omega.
\end{equation}

To present the analysis in a concise fashion, it is advisable
to introduce some vector notations and rewrite the system in a vector
format. We let 
\begin{align*}
  \qb &= (q_1,\,q_2,\,q_3)^T,\\
  \psib &= (\psi_1,\,\psi_2,\, \psi_3)^T,\\
  \phib &= (\phi_1,\,\phi_2,\, \phi_3)^T,\\
  \fb &= (f_1,\, f_2,\, f_3)^T,\\
  U &= (\ub_1,\,\ub_2,\, \ub_3)^T.
\end{align*}
The first four in the above are column vectors, while the last one
stands for a $3\times 2$ tensor, because each $\ub_i$ represents a
vector in the horizontal direction.
We designates the coefficients matrix for the zeroth order terms by  
\begin{equation*}
  L = \left(
    \begin{matrix}
      -1 & 1 & 0\\
      1 & -2 & 1\\
      0 & 1 & -1
    \end{matrix}\right).
\end{equation*}

Then the multilayer QG equations \eqref{eq:2} can be succinctly
written in the form 
of a transport equation,
\begin{equation}
  \label{eq:6}
  \dfrac{\p}{\p t}\qb + U\cdot\nabla\qb = \fb,
\end{equation}
with 
\begin{subequations}
  \label{eq:7}
  \begin{align}
    \qb &= \Delta\psib + L\psib,\label{eq:7a}\\
    U &= \nabla^\perp\psib.
  \end{align}
\end{subequations}
The boundary conditions \eqref{eq:3} and \eqref{eq:4} and the initial
conditions can also be
recast in the vector variables,
\begin{subequations}
  \label{eq:8}
  \begin{align}
    &\psib(x,t) = \bs{l}(t),\qquad \forall x\in\p\M,\label{eq:8a}\\
    &\int_\M \psib(x,t) dx = 0,
  \end{align}
\end{subequations}
and
\begin{equation}
  \label{eq:9}
  \psib(x,0) = \psib_0(x),\qquad\forall x\in\M.
\end{equation}

\section{A non-standard elliptic boundary value problem}\label{s:bvp}
When the potential vorticity $\qb$ is known, the streamfunction
$\psib$ can be recovered by solving an elliptic boundary value
problem,
\begin{equation}
  \label{eq:10}
  \left\{
    \begin{aligned}
      \Delta\psib + L\psib &= \qb, &  &x\in\M,\\
      \psib(x) &= \bs{l}, & &x\in\p\M,\\
      \int_\M \psib(x) dx &= 0. & &
    \end{aligned}\right.
\end{equation}

The boundary conditions are of a non-standard type. 
The scalar version of \eqref{eq:10} has been dealt with in REF, with
the aid of the Green's function for the Helmholtz equation. Our
strategy for the system \eqref{eq:10} is to transform and decouple it
into a sequence of scalar elliptic boundary value problems. 
We note that the coefficient matrix $L$ is symmetric, and therefore
can be diagonalized. It has a set of non-positive eigenvalues
$\{\lambda_1,\,\lambda_2,\,\lambda_3\} = \{0, \,-1,\, -3\}$, and a
corresponding set
of distinct orthogonal eigenvectors,
\begin{align*}
  &\vb_0 = \left(
    \begin{aligned}
      \dfrac{1}{\sqrt{3}}\\\dfrac{1}{\sqrt{3}}\\\dfrac{1}{\sqrt{3}}
    \end{aligned}\right), &
  &\vb_1 = \left(
    \begin{aligned}
      \dfrac{1}{\sqrt{2}}\\0\\-\dfrac{1}{\sqrt{2}}
    \end{aligned}\right), &
  &\vb_2 = \left(
    \begin{aligned}
      -\dfrac{1}{\sqrt{6}}\\\dfrac{2}{\sqrt{6}}\\-\dfrac{1}{\sqrt{6}}
    \end{aligned}\right), &
\end{align*}
corresponding to the barotropic mode, and the first and the second
baroclinic modes in the vertical direction, respectively.
Using these eigen-modes as a basis, we can transform the BVP
\eqref{eq:10} into a decoupled system. Specifically, 
We let 
\begin{equation*}
  P = [\vb_1,\, \vb_2,\, \vb_3],
\end{equation*}
and 
\begin{equation*}
  \psib = P\tilde\psib. 
\end{equation*}
For each $i=1,2,3$, the $\tilde\psi_i$ solves the boundary value
problem 
\begin{equation}\label{eq:10a}
\left\{
  \begin{aligned}
    \Delta\tilde\psi_i + \lambda_i\tilde\psi_i &= \tilde q_i, &
    &\xb\in\M,\\ 
    \tilde\psi_i &= \tilde l_i, & &\xb \in\p\M,\\
    \int_\M \psi_i(\xb) dx &= 0. & &
  \end{aligned}\right.
\end{equation}
In the above, $\tilde q_i$ and $\tilde l_i$ are obtained,
respectively, from  
 \begin{align*}
   &\tilde\qb = P^{-1} \qb, \qquad\quad\tilde{\bs{l}} = P^{-1} \bs{l}.
 \end{align*}

For  $i=2,\,3$, the non-standard scalar elliptic BVP with
a zero-order term has been
considered in \cite{Chen2017-fh}. There, with the Green's function for
the Helmholtz equation, it is shown that the constant boundary value
$|l_i|$ on the boundary can be bounded in terms of $|q_i|_\infty$, and
the solution $\psi_i$ belongs to $W^{2,p}(\M)$ for any $p > 1$, and
is H\"older and quasi-Lipschitz continuous. The case with $i=0$ in
\eqref{eq:10a} can be
handled in a similar fashion,  with the Green's function for the
Laplace operator, and the same regularity results can be
obtained. These regularity results can then be transferred to the
solution $\psib$ of \eqref{eq:10} via the transformation $P^{-1}$. 

Below, we shall formally state the regularity results for the elliptic
boundary value problem \eqref{eq:10}. But, in order to do so, we need
to first give the precise definitions of some relevant function
spaces. 

We denote by $Q_T$ the spatial-temporal domain,
\begin{equation*}
  Q_T = \M\times (0,T).
\end{equation*}
We denote by $L^\infty(\M)$, or $L^\infty(Q_T)$ when time is also
involved, the space of functions that are essentially bounded. 
We denote by $C^{0,\gamma}(\overline\M)$, with $\gamma>0$, the space of
H\"older-continuous functions on $\M$, and similarly,
$C^{0,\gamma}(\overline{Q_T})$ on $Q_T$. $C^{0,\gamma}(\overline\M)$ and
$C^{0,\gamma}(\overline{Q_T})$ are both Banach spaces under the usual H\"older
norms. 

\begin{lemma}\label{lem:elliptic-reg}
  Let $\p\M\in C^2$, $\qb\in L^\infty(\M)$. Then the elliptic boundary
  value problem \eqref{eq:10} has a unique solution $\psib\in
  W^{2,p}(\M)$ for every $p > 1$, with the following estimate,
    \begin{align}
      \|\nabla^2\psib\|_{L^p(\M)} &\leq C p \| \qb\|_{L^\infty(\M)},
      & & \forall\, p > 1.\label{eq:reg1}
\end{align}
In addition, the first derivatives of $\psib$ are H\"older 
continuous and quasi-Lipschitz continuous,
 \begin{align}
       \|\nabla\psib\|_{C^{0,\gamma}(\M)} &\leq \dfrac{C}{1-\gamma} \| \qb\|_{L^\infty(\M)},
       & & \forall\, 0<\gamma< 1,\label{eq:hold}\\
       |\nabla\psib(\xi) - \nabla\psib(\eta)|  &\leq C \chi(\delta) \| \qb\|_{L^\infty(\M)},
       & & \forall\,\xi,\,\eta\in\M.\label{eq:qlip}
     \end{align}
In the above,
\begin{equation*}
  \chi(\delta) = \left\{
    \begin{aligned}
      &(1-\ln\delta)\delta & &\textrm{if } \delta < 1,\\
      &1 & &\textrm{if } \delta \ge 1.
    \end{aligned}\right.
\end{equation*}
\end{lemma}

We denote by $V$ the space of solutions to the elliptic boundary value
problem \eqref{eq:10} with $\qb\in L^\infty(\M)$, i.e., 
\begin{equation*}
  V := \left\{ \psib\,|\, \psib \textrm{ solves } \eqref{eq:10} \textrm{
      for some }  \qb\in L^\infty(\M)\right\}.
\end{equation*}
The space $V$ is equipped with the norm
\begin{equation*}
  \| \psib\| _V := \| \Delta\psib + L\psib\|_{L^\infty(\M)}.
\end{equation*}
By the continuity of the inverse elliptic operator $(\Delta + L)^{-1}$,
$V$ is a Banach space.

In the analysis, we will also encounter functions that are
differentiable with continuous first derivatives. The space of these
functions will be denoted as $C^1(\overline\M)$, equipped with the usual $C^1$
norm. 

When time is involved, we use $L^\infty(0,T; V)$ to designate the
space of functions that are \emph{essentially} bounded with respect to
the $\| \cdot\|_V$ norm, and $L^\infty(0,T; C^1(\overline\M))$ for
functions that are \emph{essentially} bounded under the
$\|\cdot\|_{C^1(\overline\M)}$ norm. 

In the sequel, we will need the following regularity result, which can be easily
derived from the classical $L^p$ theory for elliptic equations with
Dirichlet boundary conditions (\cite{Gilbarg1983-pq}). 
\begin{lemma}\label{lem:elliptic-reg1}
  Let $\bs{g}\in L^p(\M)$ with $p>1$, and let $\psib$ be a solution of 
\begin{subequations}
  \label{eq:b16a}
  \begin{align}
    \Delta \psib + L \psib &= \sum_{i=1}^2 c_i \dfrac{\p}{\p
             x_i}\bs{g}, & \M&,\label{eq:16aa}\\ 
    \psib &= \lb, & \p\M&,\label{eq:b16ab}\\
    \int_\M\psib dx &= 0. & &\label{eq:b16ac}
  \end{align}
\end{subequations}
Then, $\psib$ has one generalized derivative, and
\begin{equation}
  \label{eq:b16b}
  \| \psib\|_{W^{1,p}(\M)} \leq Cp \|\bs{g}\|_{L^p(\M)}.
\end{equation}
Here $C$ is a constant depending on $\M$ and $c_i's$ only. 
\end{lemma}

\section{Weak formulation and {\itshape a priori}
  estimates}\label{s:weak} 
We assume that $\psib$ is a classical solution of
\eqref{eq:6}--\eqref{eq:7} subjecting 
to the constraints \eqref{eq:8}--\eqref{eq:9}. We let $\phib\in
C^\infty(Q_T)$ with $\phib|_{\p\M} = \phib|_{t=T} = 0$. We
take the inner product of \eqref{eq:6} with $\phib$ and integrate by
parts to obtain 
\begin{multline}
  \label{eq:b17}
  -\int_\M (\Delta\psib_0 + L\psib_0)\cdot\phib(x,0) dx - \int_0^T\int_\M
  (\Delta\psib + L\psib)\cdot\dfrac{\p\phib}{\p t}dx dt \\
   - \int_0^T\int_\M
  (\Delta\psib + L \psib) \cdot U \cdot\nabla\phib dxdt =
  \int_0^T\int_\M \fb\cdot \phib dxdt.
\end{multline}
Thus, every classical solution of the multilayer QG equation also
solves the integral equation \eqref{eq:b17}, but the converse is not
true, 
for the QGPV $q = \Delta\psib + L\psib$ may not be differentiable
either in space $x$ or in time $t$. Solutions of \eqref{eq:b17} are
called weak solutions of the multilayer QG. 

We establish the well-posedness of the multilayer QG
\eqref{eq:6}--\eqref{eq:9} by working with its weak formulation first, 
whose precise statement is given here.\\

\noindent{\bfseries Statement of the problem:}\\
\begin{equation}\label{problem}
  \begin{aligned}
  &\textrm{Let } \psib_0\in
  V, \textrm{ and } \fb\in L^\infty(Q_T). \textrm{ Find } \psib\in
  L^\infty(0,T; V) \textrm{ such that   
  \eqref{eq:b17} holds for every } \\
  &\phib\in C^\infty(Q_T) \textrm{
  with }
 \phib|_{\p\M} =\phib|_{t=T} = 0. 
  \end{aligned}
\end{equation}

We first obtain a few {\itshape a priori} estimates on the
solution(s) of \eqref{problem}. 
We choose $\phib(x,t) = g(t)\gammab(x)$ in \eqref{eq:b17} with
$g\in C^\infty([0,T])$, $g(T) = 0$, and  $\gammab\in C^\infty_c(\M)$.
Substituting this $\phib$ into \eqref{eq:b17}, we have
\begin{multline}
  \label{eq:b18}
 - g(0) \int_\M (\Delta\psib_0 + L\psib_0)\cdot\gammab dx - \int_0^Tg'(t)\int_\M
  (\Delta\psib + L\psib)\cdot\gammab(x)dx dt \\
   - \int_0^Tg(t)\int_\M
  (\Delta\psib + L \psib)\nabla^\perp\psib\cdot\nabla\gammab dxdt =
  \int_0^Tg(t)\int_\M f\cdot\gammab dxdt.
\end{multline}
If we take $g(0) = 0$ as well, then \eqref{eq:b18} becomes
\begin{multline}
  \label{eq:b19}
  - \int_0^Tg'(t)\int_\M
  (\Delta\psib + L\psib)\cdot\gammab(x)dx dt    = \\
   \int_0^Tg(t)\int_\M
  \left((\Delta\psib + L \psib)\nabla^\perp\psib\cdot\nabla\gammab +
    f\cdot\gammab\right) dxdt. 
\end{multline}
This shows that 
\begin{multline}
  \label{eq:b20}
  \dfrac{d}{d t}\int_\M
  (\Delta\psib + L\psib)\cdot\gammab(x)dx dt    = \\ \int_\M
  \left((\Delta\psib + L \psib)\nabla^\perp\psib\cdot\nabla\gammab + f\cdot\gammab
  \right) dx\qquad \textrm{in }\D'(0,T).  
\end{multline}
Thanks to the fact that $C^\infty_c(\M)$ is dense in $H^1_0(\M)$, the
above also holds for every $\phib\in H^1_0(\M)$. Thus, we conclude
that  $\Delta\psib + L\psib$ is weakly continuous in time in the following
sense, 
\begin{equation*}
 {\int_\M (\Delta\psib + L\psib)\cdot\gammab dx \textrm{ is continuous in time
    for every } \gammab\in H^1_0(\M).}
\end{equation*}

Integrating by parts in \eqref{eq:b18}, we find 
\begin{multline}
  \label{eq:b21}
 - g(0) \int_\M (\nabla\psib_0\cdot\nabla\gammab + L\psib_0\cdot\gammab) dx +
 \int_0^Tg'(t)\int_\M 
  (\nabla\psib\cdot\nabla\gammab + L\psib\cdot\gammab)dx dt = \\
    \int_0^Tg(t)\int_\M
  \left((\Delta\psib + L \psib)\nabla^\perp\psib\cdot\nabla\gammab +
    f\cdot\gammab\right) dxdt. 
\end{multline}
Again, taking $g(0) = 0$ yields
\begin{equation}
  \label{eq:b22}
 \int_0^Tg'(t)\int_\M 
  (\nabla\psib\cdot\nabla\gammab + L\psib\cdot\gammab)dx dt =
    \int_0^Tg(t)\int_\M
  \left( (\Delta\psib + L \psib)\nabla^\perp\psib\cdot\nabla\gammab +
    f\cdot\gammab\right) dxdt. 
\end{equation}
Since $C^\infty_c(\M)$ is dense in the space $H^1_0(\M)$ under the
usual $H^1$-norm, the above holds for every $\gammab\in
H^1_0(\M)$. Thus, 
\begin{equation}
  \label{eq:b23}
  \dfrac{d}{dt}\int_\M 
  (\nabla\psib\cdot\nabla\gammab + L\psib\cdot\gammab)dx =
    -\int_\M \left(
  (\Delta\psib + L \psib)\nabla^\perp\psib\cdot\nabla\gammab + f\cdot\gammab
  \right) dx\quad\textrm{ in } \D'(0,T).
\end{equation}
This implies that $\psib$ is weakly continuous in time for the
$H^1$-norm, 
\begin{equation*}
{\int_M(\nabla\psib\cdot\nabla\gammab + L\psib\cdot\gammab)dx \textrm{ is
    continuous in time for every } \gammab\in H^1_0(\M). }
\end{equation*}

To investigate the initial value of $\psib$, we take
$g\in\C^\infty([0,T])$ with $g(0) \ne 0$ and $g(T) = 0$. We multiply
\eqref{eq:b20} by $g(t)$ and integrate by parts in $t$ to obtain
\begin{multline}
  \label{eq:b24}
  g(0) \int_\M (\Delta\psib(x,0) + L\psib(x,0))\cdot\gammab dx - \int_0^Tg'(t)\int_\M
  (\Delta\psib + L\psib)\cdot\gammab(x)dx dt \\
   - \int_0^Tg(t)\int_\M
  (\Delta\psib + L \psib)\nabla^\perp\psib\cdot\nabla\gammab dxdt =
  \int_0^Tg(t)\int_\M f\cdot\gammab dxdt.
\end{multline}
Comparing \eqref{eq:b24} with \eqref{eq:b18}, we find that
\begin{equation}
  \label{eq:b25}
  \int_\M(\Delta\psib(x,0) + L\psib(x,0))\cdot\gammab dx = \int_\M (\Delta\psib_0 +
    L\psib_0)\cdot\gammab dx,\qquad \forall\, \gammab\in H^1_0(\M). 
\end{equation}


Multiplying \eqref{eq:b23} by the same $g(t)$ and integrating by parts
in time, we obtain
\begin{multline}
  \label{eq:b26}
 - g(0) \int_\M (\nabla\psi(x,0)\cdot\nabla\gammab + \psi(x,0)\cdot\gammab) dx +
 \int_0^Tg'(t)\int_\M 
  (\nabla\psib\cdot\nabla\gammab + L\psib\cdot\gammab)dx dt \\
    \int_0^Tg(t)\int_\M\left(
  (\Delta\psib + L \psib)\nabla^\perp\psib\cdot\nabla\gammab + f\cdot\gammab
\right) dxdt.
\end{multline}
Comparing this equation with \eqref{eq:b21}, we easily see that 
\begin{multline}
  \label{eq:b27}
  \int_\M \left(\nabla\psib(x,0)\cdot\nabla\gammab +
    L\psib(x,0) \cdot\gammab \right) dx = \\
 \int_\M \left( \nabla\psib_0\cdot\nabla\gammab +
    L\psib_0\cdot\gammab \right) dx,\quad
  \forall\,\gammab\in H^1_0(\M).
\end{multline}


We formally summarize these results in the following lemma.
\begin{lemma}\label{lem:weak-cont}
  The solution $\psib$ to the weak formulation
  \eqref{eq:b17}, if it exists, is weakly continuous in the following sense,
  \begin{subequations}
  \begin{align}
  &\int_\M (\Delta\psib + L\psib)\cdot\gammab dx \textrm{ is continuous in time
    for every } \gammab\in H^1_0(\M),\label{wc-1}\\
  &\int_M(\nabla\psib\cdot\nabla\gammab + L\psib\cdot\gammab)dx \textrm{ is
    continuous in time for every } \gammab\in H^1_0(\M).\label{wc-2}
  \end{align}
  \end{subequations}
The initial condition is satisfied in the sense that 
\begin{subequations}\label{wic}
\begin{align}
  \int_\M(\Delta\psib(x,0) + L\psib(x,0))\cdot\gammab dx &= \int_\M (\Delta\psib_0 +
    L\psib_0)\cdot\gammab dx,\qquad \forall\, \gammab\in H^1_0(\M).\label{wic1} \\
  \int_\M \left(\nabla\psib(x,0)\cdot\nabla\gammab +
    L\psib(x,0) \cdot\gammab \right) dx &= 
 \int_\M \left( \nabla\psib_0\cdot\nabla\gammab +
    L\psib_0\cdot\gammab \right) dx,\quad
  \forall\,\gammab\in H^1_0(\M).\label{wic2}
\end{align}
\end{subequations}
\end{lemma}

By virtue of Lemma \ref{lem:elliptic-reg}, any solutions of
\eqref{problem} automatically have second weak derivatives in space. In
fact, it also has second temporal-spatial cross derivatives, according
to the following lemma.
\begin{lemma}\label{lem-t-deriv}
Let $\psi(x,t)$ be a generalized solution of
\eqref{eq:6}--\eqref{eq:8} in the sense of \eqref{eq:b17}. Then there
exists generalized derivatives $\p^2\psi/\p x\p t$ and, for any $p \ge
1$, 
\begin{equation}
  \label{eq:b28}
  \sup_{0<t<T} \|\dfrac{\p^2\psib}{\p x\p t}\|_{L^p(\M)} \leq
    Cp\sup_{0<t<T} \left(\|F\|_\lpm +
      \|\psib\|_{L^\infty(0,T;V)}\cdot\|\nabla\psib\|_\lpm \right).
\end{equation}
\end{lemma}
\begin{proof}

From \eqref{eq:b20} one derives that, for a.e.~$t\in(0,\,T)$,
  \begin{equation}
    \label{eq:b29}
    (\Delta + L)\dfrac{\p}{\p t}\psib = \nabla\times F -
    \nabla\cdot\left(\nabla^\perp \psib\cdot(\Delta\psib +
      L\psib)\right) \in  H^{-1}(\M). 
  \end{equation}
Then, by Lemma \ref{lem:elliptic-reg1}, 
\begin{equation*}
  \left\|\dfrac{\p\psib}{\p t}\right\|_{\w1pm} \leq Cp\left(\|F\|_\lpm +
    \|\psib\|_{L^\infty(0,T: V)} \|\nabla\psib\|_\lpm\right).
\end{equation*}
Taking the supreme norm in time $t$ on the right-hand side, and then on
the left-hand side, we obtain 
\begin{equation*}
  \sup_{0<t<T} \left\|\dfrac{\p\psib}{\p t}\right\|_{\w1pm} \leq
  Cp\sup_{0<t<T}\left(\|F\|_\lpm + 
    \|\psib\|_{L^\infty(0,T: V)} \|\nabla\psib\|_\lpm\right).
\end{equation*}
\end{proof}

\section{Uniqueness}\label{s:unique}
In this section, we establish  the uniqueness of the
weak solution of \eqref{eq:6}--\eqref{eq:9}, if it exists. 
\begin{theorem}\label{uniqueness}
  The solution to the weak problem  \eqref{problem},
  if it exists, must be unique.
\end{theorem}

The uniqueness proof largely follows the arguments laid out by
Yudovich (\cite{Yudovich1963-bj}). What are new to the present problem
include the presence of multiple vertical layers and the
non-homogeneous boundary conditions that are needed to accommodate the
free fluctuations of the layer interfaces.

\begin{proof}
  We let $\psib^1$ and $\psib^2$ be two solutions to the weak problem
  for the same initial data $\psib_0$, and $t$ be a fixed point in
  $[0, T]$. Then, for an 
  arbitrary $\phib\in C^\infty(Q_t)$ with $\phib|_{\p\M} =
  \phib(\cdot,t) = 0$, $\psib^1$ and $\psib^2$ satisfy the following
  equations, respectively,
\begin{multline}
  \label{eq:b30}
  \int_\M (\Delta\psib_0 + L\psib_0)\cdot\phib(x,0) dx - \int_0^t\int_\M
  (\Delta\psib^1 + L\psib^1)\cdot\dfrac{\p\phib}{\p t}dx dt \\
   - \int_0^T\int_\M
  (\Delta\psib^1 + L \psib^1)\cdot\nabla^\perp\psib^1\cdot\nabla\phib dxdt =
  \int_0^T\int_\M \fb\cdot\phib dxdt,
\end{multline}
\begin{multline}
  \label{eq:b31}
  \int_\M (\Delta\psib_0 + L\psib_0)\cdot\phib(x,0) dx - \int_0^t\int_\M
  (\Delta\psib^2 + L\psib^2)\cdot\dfrac{\p\phib}{\p t}dx dt \\
   - \int_0^T\int_\M
  (\Delta\psib^2 + L \psib^2)\cdot\nabla^\perp\psib^2\cdot\nabla\phib dxdt =
  \int_0^T\int_\M \fb\cdot\phib dxdt.
\end{multline}
Subtracting these two equations, and denoting $\hb = \psib^1 - \psib^2$,
we obtain
\begin{multline}
  \label{eq:b32}
- \int_0^t\int_\M
  (\Delta \hb + L \hb)\cdot\dfrac{\p\phib}{\p t}dx dt 
   - \int_0^T\int_\M
  (\Delta \hb  + L \hb)\cdot\nabla^\perp\psib^1\cdot\nabla\phib dxdt\\
  + \int_0^T\int_\M
  (\Delta \psib^2  + L\psib^2)\cdot\nabla^\perp \hb \cdot\nabla\phib
  dxdt = 0. 
\end{multline}
An integration by parts in space in the first term leads to
\begin{multline}
  \label{eq:b32a}
\int_0^t\int_\M
  (\nabla \hb\cdot\nabla\p_t\phib -L \hb\cdot\p_t\phib)dx dt 
   - \int_0^T\int_\M
  (\Delta \hb  + L \hb)\cdot\nabla^\perp\psib^1\cdot\nabla\phib dxdt\\
  + \int_0^T\int_\M
  (\Delta \psib^2  + L\psib^2)\cdot\nabla^\perp \hb \cdot\nabla\phib dxdt = 0.
\end{multline}
Both $\phib^1$ and $\phib^2$ assume space-independent values on
the boundary $\p\M$, and so does the difference $\hb$ between
them. Thus, after a shifting in the vertical direction, $\hb$ will
vanish on the boundary. We denote this shifted function by $\hb^\#\in
L^\infty(0,t; H^1_0(\M))$. The functions $\hb$ and $\hb^\#$ are
related via 
\begin{equation}
  \label{eq:b33}
  \hb(x,\tau) = \hb^\#(x,\tau) + \lb(\tau),\qquad 0\leq \tau\leq t
\end{equation}
for some function $\lb(\tau)$. Both $\psib^1$ and $\psib^2$ have a zero
average over $\M$, and so does their difference $\hb$. Integrating
\eqref{eq:b33} over $\M$ we establish a simple relation between $\lb$
and $\hb^\#$,
\begin{equation}
  \label{eq:b34}
  \lb(\tau) = -\dfrac{1}{|\M|}\int_\M \hb^\#(x,\tau)dx.
\end{equation}

Replacing $\hb$ by $\hb^\#+\lb$ in the first and third integrals of
\eqref{eq:b32a} yields
 \begin{multline}
  \label{eq:b35}
\int_0^t\int_\M
  (\nabla \hb^\#\cdot\nabla\p_t\phib - L \hb^\#\cdot\p_t\phib)dx d\tau - \\
  \int_0^t\int_M L\lb\cdot\p_t\phib dxd\tau 
   - \int_0^T\int_\M
  (\Delta \hb^\#  + L\hb^\# + L\lb)\cdot\nabla^\perp\psib^1\cdot\nabla\phib dxd\tau \\
  + \int_0^T\int_\M
  (\Delta \psib^2  + L\psib^2)\cdot\nabla^\perp \hb^\# \cdot\nabla\phib dxd\tau = 0.
\end{multline}
We integrate by parts in time $t$ in the first integral, and use the
facts that 
\begin{align*}
  &\hb^\#(\cdot,\tau) \in H^1_0(\M), \qquad a.e.~\tau,\\
  &\int_\M (\Delta\hb + L\hb)|_{t=0}\cdot\phib(\cdot,0) dx = 0,\\
  &\int_\M (\nabla\hb(\cdot,0)\cdot\nabla\phib - L\hb(\cdot,0)\cdot\phib) dx = 0,
\end{align*}
and the regularity results from Lemma \ref{lem-t-deriv}, we obtain
 \begin{multline}
  \label{eq:b36}
-\int_0^t\int_\M
  (\p_t\nabla \hb^\#\cdot\nabla\phib - L\p_t \hb^\#\cdot\phib)dx d\tau +\\
  \int_0^t\int_M L\p_t \lb\cdot \phib dxd\tau 
   - \int_0^T\int_\M
  (\Delta \hb^\#  + L\hb^\# + L\lb )\nabla^\perp\psib^1\cdot\nabla\phib dxd\tau \\
  + \int_0^T\int_\M
  (\Delta \psib^2  + L\psib^2)\nabla^\perp \hb^\# \cdot\nabla\phib dxd\tau = 0.
\end{multline}
We note that each of the integrals is linear and continuous with
respect to $\phib$ in the norm of $L^2(0,T: H^1_0(\M))$. Thus, we
can let $\phib$ tend to $\hb^\#$ in $L^2(0,T; H^1_0(\M))$, pass to
the limit in \eqref{eq:b36}, and noticing the fact that $\nabla \hb^\#\cdot
\nabla^\perp \hb^\# = 0$, we obtain
 \begin{multline*}
-\int_0^t\int_\M
  (\p_t\nabla \hb^\#\cdot\nabla \hb^\# -  L\p_t \hb^\#\cdot \hb^\#)dx d\tau +
  \int_0^t\int_M L\p_t \lb\cdot\hb^\# dxd\tau  \\
   - \int_0^T\int_\M
  (\Delta \hb^\#  + L\hb^\# + L\lb)\cdot\nabla^\perp\psib^1\cdot\nabla
  \hb^\# dxd\tau = 0. 
\end{multline*}
\begin{multline}
\label{eq:b37}
-\int_0^t\dfrac{1}{2}\dfrac{d}{dt}\|\hb^\#\|^2_{H^1_0(\M)}dt d\tau +
  \int_0^t\int_\M L\p_t\hb^\#\cdot\hb^\#dx d\tau \\
   + \int_0^t\p_t \int_\M L\p_t\lb\cdot \hb^\# dxd\tau    -\\
   \int_0^t\int_\M 
  (\Delta \hb^\#  + L\hb^\# + L\lb)\cdot\nabla^\perp\psib^1\cdot\nabla
  \hb^\# dxd\tau = 0. 
\end{multline}
Noticing that $\lb$ is independent of $\xb$ and the no-flux boundary
conditions on $\psib^1$, we find 
\begin{align*}
&\int_0^t\int_\M L\lb\cdot\nabla^\perp\psib^1\cdot\nabla  \hb^\# dxd\tau\\
=& \int_0^t L\lb\cdot\int_\M \nabla\cdot\left(
\nabla^\perp\psib^1 \hb^\#\right) dxd\tau\\
=& \int_0^t L\lb\cdot\int_{\p\M} 
\nabla^\perp\psib^1\cdot\nb\cdot \hb^\#ds d\tau = 0.\\
\end{align*}
Using the fact that the coefficient matrix $L$ is symmetric, we can
write the second term on the left-hand side of \eqref{eq:b37} as
\begin{align*}
&\int_0^t\int_\M L\p_t\hb^\#\cdot\hb^\#dx d\tau \\
=& \int_0^t \dfrac{1}{2}\dfrac{d}{dt} \int_\M L\hb^\#\cdot\hb^\#dx d\tau \\
=& \dfrac{1}{2} \int_\M L\hb^\#(\xb,t)\cdot\hb^\#(\xb,t)dx - 
 \dfrac{1}{2} \int_\M L\hb^\#(\xb,0)\cdot\hb^\#(\xb,0)dx \\
=& \dfrac{1}{2} \int_\M L\hb^\#(\xb,t)\cdot\hb^\#(\xb,t)dx.
\end{align*}
Regarding the third term on the left-hand side of \eqref{eq:b37}, we
again use the symmetry of the linear operator $L$ and the formula
\eqref{eq:b34} to obtain
\begin{align*}
\int_0^t \int_\M L\p_t\lb\cdot \hb^\# dxd\tau =
\dfrac{-1}{2|\M|} \left(L \int_\M \hb^\#(x,t) dx\right) \cdot
\left(\int_\M \hb^\#(x,t) dx\right)
\end{align*}
Inserting these identities into \eqref{eq:b37}, we have
\begin{multline}
\label{eq:b38}
-\dfrac{1}{2}\|\hb^\#(\cdot,t)\|^2_{H^1_0(\M)} d\tau + 
\dfrac{1}{2} \int_\M L\hb^\#(\xb,t)\cdot\hb^\#(\xb,t)dx -{ }\\
\dfrac{1}{2|\M|} \left(L \int_\M \hb^\#(x,t) dx\right) \cdot
\left(\int_\M \hb^\#(x,t) dx\right)\\
  = \int_0^t\int_\M
  (\Delta \hb^\#  + L\hb^\#)\nabla^\perp\psib^1\cdot\nabla \hb^\# dxd\tau.
\end{multline}
\begin{multline}
\label{eq:b39}
\dfrac{1}{2}\|\hb^\#(\cdot,t)\|^2_{H^1_0(\M)} d\tau -
\dfrac{1}{2} \int_\M L\hb^\#(\xb,t)\cdot\hb^\#(\xb,t)dx +{ }\\
\dfrac{1}{2|\M|} \left(L \int_\M \hb^\#(x,t) dx\right) \cdot
\left(\int_\M \hb^\#(x,t) dx\right)\\
  = -\int_0^t\int_\M
  (\Delta \hb^\#  + L\hb^\#)\nabla^\perp\psib^1\cdot\nabla \hb^\# dxd\tau.
\end{multline}
We expand $\hb^\#$ in the orthonormal eigenvectors $v_0$, $v_1$ and
$v_2$ of $L$,
\begin{equation*}
  \hb^\#(x,t) = c_0(x,t) \vb_0 + c_1(x,t) \vb_1 + c_2(x,t) \vb_2.
\end{equation*}
Substituting this expansion in the second term on the left-hand side of
\eqref{eq:b39} and using the orthogonality of the eigenfunctions,  one
finds that 
\begin{align*}
  -\dfrac{1}{2} \int_\M L\hb^\#(\xb,t)\cdot\hb^\#(\xb,t)dx 
&= -\dfrac{1}{2}\left( c_1\lambda_1 \vb_1 + c_2\lambda_2 \vb_2, c_0 \vb_0 +
  c_1 \vb_1 + c_2 \vb_2\right) \\
&= -\dfrac{1}{2}\int_\M (\lambda_1 c_1^2 + \lambda_2 c_2^2) dx.
\end{align*}
Similarly, substituting this expansion for $\hb^\#$ into the third
term on the left-hand side of the equation, and using the fact that
the eigenvalues of $L$ are non-positive, one finds that
\begin{align*}
&\dfrac{1}{2|\M|} \left(L \int_\M \hb^\#(x,t) dx\right) \cdot
\left(\int_\M \hb^\#(x,t) dx\right)\\  
=& \dfrac{1}{2|\M|} \left(\lambda_1 \left(\int_\M c_1(x,t) dx\right)^2
   + \lambda_2 \left(\int_\M c_2(x,t) dx\right)^2\right)\\
\ge& \dfrac{1}{2}\int_\M (\lambda_1 c_1^2 + \lambda_2 c_2^2) dx.
\end{align*}
Thus the combination of the second and third terms on the left-hand
side of \eqref{eq:b39} is positive, and the following
inequality results,
\begin{equation}
  \label{eq:b40}
 \dfrac{1}{2} \|\hb^\#(\cdot,t)\|_{H^1_0(\M)}^2 \leq - \int_0^t\int_\M
 (\Delta \hb^\# + L\hb^\#)\cdot \nabla^\perp \psib^1\cdot\nabla \hb^\# dx
 d\tau. 
\end{equation}

Unlike in the barotropic case (see \cite{Chen2017-fh}), the inner product involving the
zero order term on the right-hand side does not vanish, thanks to  the
vertical layer interactions. We proceed by obtaining 
an estimate on this term. Using the fact that the linear operator $L$
is a constant 
coefficient matrix and $\nabla^\perp\psib^1$ is H\"older continuous on
$\M$, thanks to Lemma \ref{lem:elliptic-reg}, we obtain that
\begin{equation}\label{eq:41}
\left\vert \int_0^t\int_\M
  L\hb^\#\cdot \nabla^\perp \psib^1\cdot\nabla \hb^\# dx
 d\tau \right\vert
\leq C(\M) |U^1|_\infty \int_0^t|\nabla \hb^\#|^2_{L^2(\M)} d\tau.
\end{equation}
 For the other term on the right-hand
side of \eqref{eq:b40}, we first notice that
$\nabla^\perp\psib^1\cdot\nabla\hb^\# = 0$ on the boundary. Then, by
an integration by 
parts, one obtains 
\begin{equation*}
\int_0^t\int_\M
 \Delta \hb^\# \cdot \nabla^\perp \psib^1\cdot\nabla \hb^\# dx
 d\tau = -\int_0^t\int_\M
 \nabla \hb^\#\cdot \nabla\cdot( \nabla^\perp\psib^1\cdot\nabla \hb^\# )dx d\tau 
\end{equation*}
To further investigate the integral on the right-hand side, we
introduce index $i, j = 1,\,2$ for the horizontal coordinates, and
$l=1,\,2,\,3$ for the vertical layers, and the notations
\begin{equation*}
  (u_{l1},\,u_{l2}) \equiv \nabla^\perp \psib^1 \equiv
  (-\p_2\psib^1_l,\,\p_1\psib^1_l). 
\end{equation*}
Then, using the Einstein convention of repeated indices for $1\le
i,\,j\le 2$, we proceed with the derivation,
\begin{align*}
&\int_0^t\int_\M\Delta \hb^\# \cdot \nabla^\perp \psib^1\cdot\nabla \hb^\# dx
 d\tau \\
=& -\int_0^t\int_\M \sum_{l=1}^3 
 \p_i h^\#_l \p_i(u^1_{lj}\p_j h^\#_l)dx \tau \\
=& -\int_0^t\int_\M \sum_{l=1}^3 
 \p_i h^\#_l (\p_i u^1_{lj}\p_j h^\#_l + u^1_{lj} \p_j\p_i h^\#_l )dx
   \tau \\
=& -\sum_{l=1}^3 \int_0^t\int_\M 
 \p_i h^\#_l \p_i u^1_{lj}\p_j h^\#_l + \dfrac{1}{2}u^1_{lj} \p_j(\p_i h^\#_l)^2 )dx \tau \\
=& -\sum_{l=1}^3 \int_0^t\int_\M 
 \p_i h^\#_l \p_i u^1_{lj}\p_j h^\#_l dx \tau.
\end{align*}
The last step results thanks to the fact that $(u^1_{l1},\,u^1_{l2})$
has a zero normal component along the boundary. Reintroducing the
expressions for $u_{li}$ in the above, and after rearrangements, we
find 
\begin{align*}
  &\int_0^t \int_\M \Delta \hb^\# \nabla^\perp\psib^1\cdot\nabla
    \hb^\# dx \\
= & -\sum_{l=1}^3 \int_0^t\int_\M
  \left\{ \left[-(\p_1 h_l^\#)^2 + (\p_2 h_l^\#)^2\right]\p_1\p_2\psi_l^1 +
    \p_1 h^\#_l\p_2 h_l^\#(\p_1^2\psi_l^1 - \p_2^2\psi_l^1)\right\} dx.
\end{align*}
Taking the absolute value on both sides, one obtains
\begin{equation}\label{eq:42}
  \left\vert \int_0^t \int_\M \Delta \hb^\# \nabla^\perp\psib^1\cdot\nabla
    \hb^\# dx\right\vert \leq C \int_0^t\int_\M |\nabla\hb^\#|^2 \cdot
|\nabla^2\psib^1|dx d\tau.
\end{equation}
Here $C$ is a constant independent of $\hb^\#$, $\psib^1$, $\M$, or
$q_0$. Combining \eqref{eq:b40}, \eqref{eq:41}, and \eqref{eq:42} , we
derive that 
\begin{multline}
  \label{eq:43}
 \dfrac{1}{2} \|\hb^\#(\cdot,t)\|_{H^1_0(\M)}^2 \leq C \int_0^t\int_\M |\nabla\hb^\#|^2 \cdot
|\nabla^2\psib^1|dx d\tau + \\
C(\M)|U^1|_\infty \int_0^t|\nabla \hb^\#|^2_{L^2(\M)} d\tau.
\end{multline}
The second-order derivatives of $\psib^1$ are not uniformly bounded
over $\M$; they belong to $L^p(\M)$ for any $p > 1$ according to the
regularity result Lemma \ref{lem:elliptic-reg}. They can be handled
with the technique employed in \cite{Yudovich1963-bj, Chen2017-fh}. 
We  denote
\begin{align*}
  M_1 \equiv & \sup_{0<t<T} \|\psib^1(\cdot,t)\|_V,\\
  M_2 \equiv & \sup_{0<t<T} \|\hb(\cdot,t)\|_V.
\end{align*}
It is then inferred from Lemma \ref{lem:elliptic-reg} that
\begin{align*}
&|U^1|_\infty \leq C(\M) M_1,\\
  &\sup_{0<t<T} \| \psib^1(\cdot,t)\|_{W^{2,\frac{2}{\epsilon}}(\M)}
  \leq C\dfrac{2}{\epsilon} M_1,\\
  &\sup_{0<t<T} \| \nabla \hb^\# (\cdot,t)\|_{\linfm}
  \leq C M_2,
\end{align*}
We let $\epsilon >0$ be arbitrary, and using the H\"older's
inequality, we derive an estimate for the first integral on
the right-hand side of \eqref{eq:43},
\begin{align*}
&C \int_0^t\int_\M |\nabla\hb^\#|^2 \cdot
|\nabla^2\psib^1|dx d\tau  \\
\leq &  C |\nabla \hb^\#|^\epsilon_{L^\infty(Q_T)} \int_0^t \int_\M |\nabla
  \hb^\#|^{2-\epsilon} \cdot |\nabla^2\psib^1| dx d\tau\\
\leq &  C |\nabla \hb^\#|^\epsilon_{L^\infty(Q_T)} \int_0^t 
\left(\int_\M |\nabla
  \hb^\#|^{(2-\epsilon)\cdot\frac{2}{2-\epsilon}}dx\right)^{\frac{2}{2-\epsilon}}
       \left(\int_\M  |\nabla^2\psib^1|^{\frac{2}{\epsilon}} dx\right)^{\frac{\epsilon}{2}} d\tau\\
\equiv &  C |\nabla \hb^\#|^\epsilon_{L^\infty(Q_T)} \int_0^t 
   \|\nabla\hb^\#(\cdot,\tau)\|^{2-\epsilon}_{\l2m} \cdot
   \|\psib^1(\cdot,\tau)\|_{W^{2,\frac{2}{\epsilon}}(\M)} d\tau\\
\leq & \dfrac{C(\M)M_1 M_2^\epsilon}{2} \int_0^t 
   \|\nabla\hb^\#(\cdot,\tau)\|^{2-\epsilon}_{\l2m} d\tau.
\end{align*}
Applying this estimate in \eqref{eq:43} yields
\begin{multline}
\label{eq:b44}
\| \hb^\#(\cdot,t)\|^2_{H^1_0(\M)} 
    \leq 
   C(\M) M_1 \left( \dfrac{M_2^\epsilon}{\epsilon} \int_0^t
   \|\hb^\#(\cdot,\tau)\|^{2-\epsilon}_{\h10m} d\tau\right. + \\
\left.\int_0^t \|\hb^\#(\cdot,\tau)\|^{2}_{\h10m} d\tau.\right)
\end{multline}

We denote 
\begin{equation*}
  \sigma(\cdot,t) \equiv \| \hb^\#(\cdot, t)\|_{\h10m}.
\end{equation*}
Then \eqref{eq:b44} can be written as 
\begin{equation}
  \label{eq:b45}
  \sigma^2(t) \leq C(\M, M_1)\left( \dfrac{M^\epsilon_2}{\epsilon}\int_0^t
  \sigma^{2-\epsilon}(\tau) d\tau + \int_0^t \sigma^2(\tau)d\tau\right)
\end{equation}
An estimate on $\sigma$ can be obtained by  the Gronwall
inequality. Indeed, denoting the right-hand side as $F(t)$, taking its
derivative, one has
\begin{equation*}
  \dfrac{d}{dt} F(t) \leq \dfrac{C M^\epsilon_2}{\epsilon}
  F^{1-\frac{\epsilon}{2}}(t) + CF(t). 
\end{equation*}
An integration of this inequality yields
\begin{equation*}
  F(t) \leq e^{2Ct} M_2^2 \left(\dfrac{C t}{2}\right)^{\frac{2}{\epsilon}}.
\end{equation*}
Thus,
\begin{equation}
  \label{eq:b46}
  \| \hb^\#(\cdot,t)\|^2_{\h10m} \leq F(t) \leq e^{2Ct} M_2^2
  \left(\dfrac{C t}{2}\right)^{\frac{2}{\epsilon}}. 
\end{equation}
We take 
\begin{equation*}
  t^\ast = \dfrac{1}{C}.
\end{equation*}
Then, for $0\leq t\leq t^\ast$, 
\begin{equation}
  \label{eq:b47}
  \| \hb^\#(\cdot,t)\|^2_{\h10m} \leq
e^{2Ct^\ast} M_2^2 \left(\dfrac{1}{2}\right)^{\frac{2}{\epsilon}}.   
\end{equation}
This estimate holds for arbitrary $\epsilon > 0$. Thus, $  \|
\hb^\#(\cdot,t)\|^2_{\h10m} $ must vanish for $0\leq t\leq t^\ast$. The
process can be repeated over subsequent time intervals of length
$t^\ast$, and thus $  \| \hb^\#(\cdot,t)\|^2_{\h10m} = 0$ for the whole
time interval $[0,\,T]$. Combined with the relations \eqref{eq:b33} and
\eqref{eq:b34}, it implies that 
\begin{equation*}
  \hb(\cdot,t) = 0\qquad \textrm{for a.e. } 0\leq t\leq T.
\end{equation*}
Thus it has been proven that the  solution to the weak problem
\eqref{problem}, if it exists, 
must be unique.
\end{proof}

\section{Existence of a solution to the weak problem}\label{s:exist} 
Yudovich (\cite{Yudovich1963-bj}) establishes the existence of a weak
solution to the two-dimensional Euler equation through an iterative
scheme and the Schauder fixed-point theorem. The existence of a
solution to the linearized problem is achieved via a regularization
technique. Here, we largely follow the footsteps this work, but we
treat the linearized problem with a flow map constructed out of a
continuous velocity filed.

Given $\psib\in
L^\infty(0,T; C^1(\Mbar))$, we compute the updated $\psib'\in
L^\infty(0,T; V)$ via
\begin{multline}
  \label{eq:b48}
  -\int_\M (\Delta\psib_0 + L\psib_0)\cdot\phib(\xb,0) dx - \int_0^T\int_\M
  (\Delta\psib' + L\psib')\cdot\dfrac{\p\phib}{\p t}dx dt \\
   - \int_0^T\int_\M (\Delta\psib' + L
   \psib')\cdot\nabla^\perp\psib\cdot\nabla\phib dxdt = 
  \int_0^T\int_\M \fb\cdot\phib dxdt,
\end{multline}
for every $\phib\in C^\infty(Q_T)$ with $\phib|_\M =
\phib(\cdot,T) = 0$. 
If a solution $\psib'$ exists, then the weak problem
\eqref{eq:b48} defines a mapping 
\begin{equation*}
  \mathcal{S}:\, L^\infty(0,T; C^1(\Mbar) \longrightarrow
  L^\infty(0,T; V) \subset L^\infty(0,T; C^1(\Mbar)).
\end{equation*}
The plan then is to show that this map has a fixed point. Since
$\mathcal{S}$ maps $L^\infty(0,T; C^1(\Mbar))$ into $L^\infty(0,T;
V)$, this fixed point is a solution of the original weak problem
\eqref{problem}.  

\begin{lemma}\label{lem:exst-l}
 Let $T>0$ and $\fb\in L^\infty(Q_T)$. Then for each  $\psib\in
 L^\infty(0,T; C^1(\Mbar))$, the weak problem 
 \eqref{eq:b48} has at least one solution $\psi'\in L^\infty(0,T; V)$. 
\end{lemma}

The basic idea for the proof of Lemma \ref{lem:exst-l} is to show that
its solution is the weak solution of the transport equation
\begin{equation}
  \label{eq:b50}
  \left\{
    \begin{aligned}
      &\dfrac{\p}{\p t}\qb' + \ub\cdot\nabla \qb' = f,\\
      & \qb'(\xb,0) = \qb_0(\xb).
    \end{aligned}\right.
\end{equation}
The weak solution of \eqref{eq:b50} can be constructed with a flow map
$\Phi_t(\ab)$ determined from the velocity field $\ub$,
\begin{equation}
  \label{eq:b52}
  \left\{
    \begin{aligned}
      &\dfrac{d}{dt}\Phi_t(\ab) = \ub(\Phi_t(\ab), t),\qquad t>0,\\
      &\Phi_0(\ab) = \ab.
    \end{aligned}\right.
\end{equation}
We now show that, given an continuous velocity field within the domain
$\M$, there exists at least one flow map $\Phi_t(\ab)$ satisfying
\eqref{eq:b52}. 

\begin{lemma}\label{lem:flow}
Assume that $\p\M\in C^2$, and let $T>0$ be arbitrary and $\ub\in
L^\infty(0,T; C(\Mbar))$ with 
$\ub\cdot\nb = 0$ on $\p\M$. Then the initial value problem
\eqref{eq:b52} has at least one solution $\Phi_t(\ab)$ that is valid
over $[0,\,T]$. 
\end{lemma}
\begin{proof}
  For each interior point, a solution can be constructed by the Peano
  method. The solution can be extended by the same method as long as
  it has not reached the boundary $\p\M$. Hence the proof is complete
  once it is shown that, starting from the a point on the boundary,
  there exists at least one solution $\Phi_t(\ab)$ that remains on
  $\p\M$ for all time. 

The boundary $\p\M$ is $C^2$ smooth. Then, locally, it can be
parameterized by a single parameter $\tau\in I$. We let $\bs{b}(\tau)$
be a vector-valued function representing the boundary. By assumption,
$\bs{b}(\tau)\in C^2(I)$. If this parameterization of $\p\M$ is only
local, then one can cover the entire $\p\M$ with a finite number of
patches, each of which is parameterized by a single parameter. We now
show that, starting from any point $\ab\in\bs{b}(I)$, there exists at
least one solution $\Phi_t(\ab)$ that remains in $\bs{b}(I)$ either
for all time   or until it exits from one of the end points of
$\bs{b}(I)$. The velocity field on the boundary can also be expressed
using the parameter $\tau$, 
\begin{equation*}
  \ub = \ub(\tau, t). 
\end{equation*}
By assumption, $\ub(\tau, t)$ is parallel to the tangential vector on
the boundary, and the following relation holds,
\begin{equation*}
 \ub(\tau, t) = \sigma(\tau, t)\bs{b}'(\tau),
\end{equation*}
for some scalar function $\sigma(\tau, t)$  that is continuous in
$\tau$, and bounded in $t$. We look for a solution of \eqref{eq:b52}
in the form of 
\begin{equation*}
  \Phi_t(\tau_0) = b(\tau(\tau_0, t)).
\end{equation*}
We let $\tau(\tau_0,t)$ be such that 
\begin{equation}\label{eq:44}
  \left\{
    \begin{aligned}
      &\tau(\tau_0, 0) = \tau_0,\\
      &\dfrac{d}{dt} \tau(\tau_0, t) = \sigma(\tau(\tau_0, t), t). 
    \end{aligned}\right.
\end{equation}
Then it is
easy to check that $\Phi_t(\tau_0)$ solves the initial-value problem
\eqref{eq:b52}, for 
\begin{equation*}
    \begin{aligned}
      &\Phi_0(\tau_0) = \bs{b}(\tau_0),\\
      &\dfrac{d}{dt}\Phi_t(\tau_0) =
      \bs{b}'(\tau(\tau_0,t))\cdot\sigma(\tau(\tau_0, t),
      t)=\ub(\tau(\tau_0,t), t). 
    \end{aligned}
\end{equation*}
The existence of a solution for the duration of $\tau(\tau_0,t)\in I$
is just another application of the Peano method. 
\end{proof}

We now prove the existence of a solution to the linearized equation
\eqref{eq:b48}.
\begin{proof}[Proof of Lemma \ref{lem:exst-l}]
We let $\Phi_t(\cdot)$ be a global solution of the initial value
problem \eqref{eq:b52}. We define a new QGPV $q'$
from the given initial state $q_0$ via
\begin{equation}
  \label{eq:b53}
  \qb'(\xb,t) = \qb_0(\Phi_{-t}(\xb)) + \int_0^t \fb(\Phi_{s-t}(\xb),s)ds,
  \quad \forall \xb\in\M,\,t>0,
\end{equation}
or equivalently, 
\begin{equation}
  \label{eq:b53a}
  \qb'(\Phi_t(\ab),t) = \qb_0(\ab) + \int_0^t \fb(\Phi_{s}(\ab),s)ds,
  \quad \forall \ab\in\M,\,t>0.
\end{equation}
It is clear that, given $\fb\in L^\infty(Q_T)$, $\qb'\in
L^\infty(Q_T)$. 
We now verify that $\qb'(\xb,t)$ defined in \eqref{eq:b53} solves the
transport equation \eqref{eq:b50} in the weak sense. We denote $\ub
\equiv \nabla^\perp\psib$, and let $\phib\in
C^\infty(Q_T)$ with $\phib|_{\p\M} = \phib(\cdot, T) = 0$. Using the
fact that the map $\xb = \Phi_t(\ab)$ is area preserving, we derives that
\begin{align*}
  &\int_0^T\int_\M q'(\xb,t)\left(\dfrac{\p\phib}{\p t} +
    \ub\cdot\nabla\phib\right) dxdt \\
  =& \int_0^T\int_\M \left(q_0(\ab) + \int_0^t
    f(\Phi_s(\ab),s)ds\right) \dfrac{d}{dt}\phib(\Phi_t(a),t)dadt \\
  =& -\int_\M q_0(\ab)\phib(\ab,0)da - \int_0^T\int_\M
  f(\Phi_t(\ab),t)\phib(\Phi_t(\ab),t)dadt \\
  =& -\int_\M q_0(\xb)\phib(\xb,0)dx - \int_0^T\int_\M
  f(\xb,t)\phib(\xb,t) dx dt.
\end{align*}
Hence, we have shown that
\begin{multline*}
-\int_\M q_0(\xb)\phib(\xb,0)dx - \int_0^T\int_\M q'(\xb,t)\left(\dfrac{\p\phib}{\p t} +
    \ub\cdot\nabla\phib\right) dxdt 
  =  \\
\int_0^T\int_\M
  f(\xb,t)\phib(\xb,t) dx dt.
\end{multline*}
We let $\psib'$ be a solution to the boundary value problem
\eqref{eq:10} corresponding to the QGPV $\qb'$. Then $\psib'\in
L^\infty(0,T; V)$ and it is a solution to the weak problem
\eqref{eq:b48}. 
\end{proof}

\begin{lemma}\label{lem-estim-l}
  The solution $\psi'$ to the weak problem has the following
  estimates,
  \begin{align}
    |\psib'|_{L^\infty(0,T; V)} &\leq |\psib_0|_V + |\fb|_{L^1(0,T;
      L^\infty(\M))}, \label{eq:b62}\\
   \max_{0\leq t\leq T}\left|\dfrac{\p\psib'}{\p
       t}\right|_{W^{1,p}(\M)} &\leq C p\max_{0\leq t\leq T}\left(
     |F|_{L^p} + |\Delta\psib' + L\psib'|_{L^\infty(Q_T)}
     |\nabla\psib|_{L^p(\M)}\right).\label{eq:b63} 
  \end{align}
\end{lemma}

\begin{proof}
  We start from the equation \eqref{eq:b53},
  \begin{align*}
    |q'(\cdot,t)|_{L^\infty(\M)} &\leq |q_0|_{L^\infty(\M)} + \int_0^t
    |f(\cdot,s)|_{L^\infty(\M)} ds,\\
    |q'|_{L^\infty(Q_T)} &\leq |q_0|_{L^\infty(\M)} + \int_0^T
    |f(\cdot,s)|_{L^\infty(\M)} ds\\
    &= |q_0|_{L^\infty(\M)} + |f|_{L^1(0,T; L^\infty(\M))}. 
  \end{align*}
The inequality \eqref{eq:b62} follows. 

The inequality \eqref{eq:b63} can be established in a similar way as in
Lemma \ref{lem-t-deriv}. 
\end{proof}

\begin{lemma}\label{lem-compact}
  The mapping 
  \begin{equation*}
    \psib \longrightarrow \psib'\qquad\textrm{ in }L^\infty(0,T; C^1(\overline\M))
  \end{equation*}
is compact.
\end{lemma}

\begin{proof}
  We let $\{\psib_r\}$ be a bounded sequence in $L^\infty(0,T;
  C^1(\overline\M))$, and let $\psib_r'$ be the corresponding sequence of
  solutions of \eqref{eq:b48}. Then, by \eqref{eq:b62} and
  \eqref{eq:b63}, we obtain, for $\forall\,p>1$, 
  \begin{equation*}
    \max_t \left|\dfrac{\p^2\psib'_r}{\p t\p x}\right|_{L^p(\M)} \leq C,
  \end{equation*}
where $C$ is a constant independent of the index $r$. Thus
$\p\psi'_r/\p x$ has one generalized derivative in both $t$ and $x$,
and 
\begin{equation*}
 \left\|\dfrac{\p\psib'_r}{\p x}\right\|_{W^{1,p}(Q_T)} \leq C,\qquad
 \forall\, p> 1.
\end{equation*}
For $p>3$, and by the Sobolev embedding theorem, $\p\psi'_r/\p x$ is
Hold\"er continuous in the temporal-spatial domain $Q_T$, and
\begin{equation*}
 \left\|\dfrac{\p\psib'_r}{\p x}\right\|_{C^{0,\lambda}(Q_T)} \leq C,
\end{equation*}
for some $0<\lambda < 1$. This shows that $\p\psib'_r/\p x$ are
equi-continuous in $Q_T$, and so is $\psi'_r$. By the Arzel\'a-Ascoli
theorem, there exists a subsequence, still denoted by the index $r$,
such that 
\begin{align*}
  \psib'_r &\longrightarrow \psib,\\
  \dfrac{\p\psib'_r}{\p x} &\longrightarrow \phib.
\end{align*}
Due to the completeness of the Banach space $L^\infty(0,T; C^1(\overline\M))$,
we have that $\psib\in L^\infty(0,T; C^1(\overline\M))$ and 
$\phib = \p\psib/\p
x$. 
\end{proof}

\begin{theorem}\label{thm-existence}
  There exists a solution to the weak problem \eqref{problem} in
  $L^\infty(0,T; V)$. 
\end{theorem}
\begin{proof}
  The mapping 
  \begin{equation*}
    \mathcal{S}:\, \psib\longrightarrow \psi'
  \end{equation*}
is compact in $L^\infty(0,T; C^1(\overline\M))$. By Schauder's fixed point
theorem, it has a fixed point $\psib$ in the same function space. Since
$\mathcal{S}$ maps from 
$L^\infty(0,T; C^1(\overline\M))$ into $L^\infty(0,T; V)$, the fixed point
$\psib$ belongs to $L^\infty(0,T; V)$ as well. 
\end{proof}

\begin{theorem}\label{thm:reg}
  Let $\fb = \nabla\times F$ be bounded and $F$ be strongly continuous
  in time $t$. Then 
  the initial and boundary conditions \eqref{eq:6}--\eqref{eq:9} are
  satisfied  in the classical sense, and $\Delta\psib$, $\p^2\psi/\p
  x\p t$ are strongly continuous with respect to $t$ on $[0,T]$ in
  $\lpm$ for any $p>1$. 
\end{theorem}
\begin{proof}
We first show that the QGPV $\qb(\cdot,t)$ is continuous in $t$ for
any $L^p(\M)$ norm with $p>1$. This improves over \eqref{wc-1} of
Lemma \ref{lem:weak-cont}. Starting from \eqref{eq:b20} and by a
well-known result (\cite{Temam2001-th}, Lemma 1.1 of Section 3.1), we
derive that, 
for some $0\leq \tau_1<\tau_2\leq T$,
  \begin{equation*}
    (\qb(\cdot,\tau_2),\,\phib) - (\qb(\cdot,\tau_1),\,\phib) =
    \int_{\tau_1}^{\tau_2} \int_\M 
    (\qb\nabla^\perp\psib \cdot \phib + \fb\cdot\phib) dx dt.
  \end{equation*}
We note that $\psib\in L^\infty(0,T; \,V)$, $q$ is bounded in
$L^\infty(Q_T)$, and $\nabla^\perp\psib$ is uniformly bounded in
$Q_T$. Thus, as $\tau_2 \longrightarrow \tau_1$, the right-hand side
vanishes, and one has
\begin{equation}
  \label{eq:b64}
  \qb(\cdot,\tau_2) \rightharpoonup \qb(\cdot,\tau_1)\qquad\textrm{in any }
  L^p(\M). 
\end{equation}

Similar to \eqref{eq:b62}, one can derive that, for $\forall\,p>1$, 
\begin{equation*}
  |\qb(\cdot,\tau_2)|_{\lpm}\leq |\qb(\cdot,\tau_1)|_\lpm +
  \int_{\tau_1}^{\tau_2} |f(\cdot,t)|_\lpm dt.
\end{equation*}
From this estimate we conclude that 
\begin{equation*}
  \overline\lim_{\tau_2\rightarrow\tau_1} |\qb(\cdot,\tau_2)|_\lpm \leq
  |\qb(\cdot,\tau_1)|_\lpm. 
\end{equation*}
By the Radon-Riesz theorem, $\qb(\cdot,\tau_2)$ converges to
$\qb(\cdot,\tau_1)$, as $\tau_2\rightarrow\tau_1$, in the strong norm of
$\lpm$. Hence, $\qb(\cdot,t)$ is continuous in $t$ in any $\lpm$, 
\begin{equation}
  \label{eq:45a}
  \qb \in \C([0,T]; \lpm),
\end{equation}
which
implies that the initial condition \eqref{eq:9} is satisfied in a
stronger norm,
\begin{equation}
  \label{eq:b65}
  \qb(\cdot,0) = q_0(\cdot) \qquad \lpm,\,\,\forall\, p> 1.
\end{equation}

Concerning the continuity of $\p^2/\p x\p t$, we derive from
\eqref{eq:6} that
\begin{equation*}
    (\Delta + L) \dfrac{\p}{\p t}\psib =  \nabla\times F -
    \nabla\cdot\left(\nabla^\perp\psib(\Delta\psib + L
      \psib)\right)\qquad \textrm{in the distribution sense}. 
\end{equation*}
Thus, formally, one has
\begin{equation*}
    \nabla \dfrac{\p}{\p t}\psib =  \nabla(\Delta+L)^{-1}\nabla\times F -
    \nabla(\Delta + L)^{-1}\nabla\cdot\left(\nabla^\perp\psi(\Delta\psib + L
      \psib)\right),  
\end{equation*}
where $(\Delta + L)^{-1}$ is the solution operator of the elliptic
boundary value problem \eqref{eq:10}. We note that $q \equiv \Delta\psib +
L\psib$ is continuous in $t$ in any $\lpm$ with $p>1$, and 
$\nabla^\perp\psib$ is uniformly bounded in $Q_T$. Thus, thanks to the
continuity of the differential operator $\nabla(\Delta -
I)^{-1}\nabla(\cdot)$, and the continuity of  $F$ with respect to $t$,
$\nabla\frac{\p\psib}{\p t}$ is continuous in $t$ in any $\lpm$ 
with $p>1$. 

By Lemma \ref{lem:elliptic-reg},
\begin{equation*}
  |\psib(\cdot,\tau_2) - \psib(\cdot,\tau_1)|_{W^{2,p}(\M)} \leq Cp
  |\qb(\cdot,\tau_2) - \qb(\cdot,\tau_1) |_\lpm. 
\end{equation*}
Thus, as $\tau_2\longrightarrow \tau_1$, 
\begin{equation*}
  \psib(\cdot,\tau_2) \longrightarrow \psib(\cdot,\tau_1)\qquad
  \textrm{in } W^{2,p}(\M),
\end{equation*}
This shows that the initial condition \eqref{eq:9} holds in $W^{2,p}(\M)$,
\begin{equation*}
  \psib(\cdot,0) = \psib_0(\cdot)\qquad\textrm{ in } W^{2,p}(\M). 
\end{equation*}

We also note that $\psib\in L^\infty(0,T; V)$ implies that
\begin{equation}
  \label{eq:b67}
  \dfrac{\p\psib}{\p x} \in L^\infty(0,T; W^{1,p}(\M)).
\end{equation}
From Lemma \ref{lem-t-deriv}, we have
\begin{equation}
  \label{eq:b68}
  \dfrac{\p^2\psib}{\p t\p x} \in L^\infty(0,T; \lpm) \subset L^p(Q_T).
\end{equation}
Combining \eqref{eq:b67} and \eqref{eq:b68}, we derive that 
\begin{equation}
  \label{eq:b69}
  \dfrac{\p\psib}{\p x} \in W^{1,p}(Q_T),\qquad\forall\, p> 1. 
\end{equation}
We take a $p > 3$. Then, by the Sobolev imbedding theorem,
\begin{equation}\label{eq:48}
  \dfrac{\p\psib}{\p x} \in C^{0,\lambda} (Q_T)\qquad \textrm{for some
  } 0 < \lambda < 1.
\end{equation}
Thus, the streamfunction $\psib$ is continuous in the spatial-temporal
domain, and the initial and boundary conditions are satisfied in the
classical sense. 

Finally, \eqref{eq:45a}, together with \eqref{eq:48}, implies that
\begin{equation*}
  \Delta\psib \in \C([0,T]; \lpm).
\end{equation*}
\end{proof}

We note that $\qb$ assumes its initial value $\qb_0$ in the
$L^p$-norm ($\forall\,p>1$), which is an improvement over
\eqref{wic1}.

\section{Concluding remarks}\label{s:conclude}
As far as model complexity is concerned, the suite of QG models sit
between the purely planar two-dimensional Euler/Navier-Stokes
equations and the fully three-dimensional Euler/Navier-Stokes
equations. Even though the three-dimensional QG equation
is posed over a three-dimensional spatial domain, the velocity vector
at every point in the space is assumed to be horizontal,
i.e.~two-dimensional. The theory on the two-dimensional planar flow is
rather complete, see, among many other references,
\cite{Yudovich1963-bj, Bardos1972-ar, Ladyzhenskaya1969-lp}. On the
other hand, the theory about the three-dimensional Euler or NSEs is
rather incomplete. For example, it remains an open question whether
the three-dimensional NSEs is globally well-posed (\cite{Temam2001-th,
  Constantin1993-jl}. The situation with the inviscid model (3D Euler)
is generally worse (\cite{Bardos2007-cf, Constantin2007-ob,
  Beale1984-we, Castro2016-gw}). It is then natural to ask how the QG equations
fare as far as the well-posedness is concerned. The global
well-posedness of the single-layer barotropic QG has been established
in a previous study (\cite{Chen2017-fh}). The current work deals with
the multi-layer QG equations. The multi-layer QG can be viewed as a
stack of single-layer barotropic equations. The layers interact with
each other through pressure. The layer interactions add some vertical
variations to the model, and move it one step closer to the full
three-dimensional fluid model. The added vertical variations make the
well-posedness issue more interesting and more challenging at the same
time.  Notwithstanding the technical challenges, the inviscid
multi-layer QG model is shown 
to be globally well-posed.
In what follows, we briefly review the challenge and the general
approach of the present work.

The presence of multiple layers and the interactions among them are
dealt with by a combination of two techniques: mode decomposition and
straightforward estimation. When the positive (or negative)
definiteness is needed, a decomposition in the barotropic and
baroclinic modes, which are the eigenmodes of the associated
Sturm-Liouville problem, is employed. In other places, i.e.~within the
nonlinear terms, a straightforward estimation is carried out on the
streamfunctions across all layers. The last technique works, of
course, thanks to the finiteness in the number of layers and the
consequential boundedness in the thickness of the layers. 

Concerning the general approach of the analysis, in establishing the
global well-posedness of the barotropic QG 
equation, Chen \cite{Chen2017-fh} follows \cite{Marchioro1994-yt}, and
employs the Picard iterations to prove the existence and uniqueness of
the flow map and the convergence of the iterative scheme. The downside
of this approach is that it has a higher regularity requirement on
the right-hand side forcing, which needs to be uniformly continuous in
space. Yudovich \cite{Yudovich1963-bj} establishes the convergence of
the iterative scheme by the Schauder fixed-point theorem, and the
right-hand side forcing is only assumed to be essentially bounded. In
order to remove the somewhat stringent requirement on the forcing in
\cite{Chen2017-fh}, the current work adopts the approach of
\cite{Yudovich1963-bj}, but with one modification. The existence of a
solution to the linearized problem is established by a flow map
corresponding to a continuous velocity field, instead of the
regularization technique employed by Yudovich. The flow map of the
merely continuous velocity field is not unique, and therefore the
solution to the linearized problem is not unique either
(\cite{DiPerna1989-fg, Alberti2011-hm, Alberti2014-hc}). Fortunately,
the uniqueness is not required by the Schauder fixed-point
theorem. The uniqueness of the weak solution to the nonlinear QG
equation has been established by the {\itshape a priori} estimates.

The next target for the current project is naturally the
three-dimensional QG equation. The equation represents a giant step
from the barotropic or the multi-layer QG equations. It is a
three-dimensional model with an infinite number of degrees of freedom
in the vertical direction. The techniques of the current and the
previous work (\cite{Chen2017-fh}), adept at treating a single layer,
or a finite number of layers, will probably be inadequate for the
three-dimensional model. It is expected that new techniques will have
to be introduced or invented. Progresses in this regard will be
reported elsewhere.

\section*{Acknowledgment}

\bibliographystyle{AIMS}
\bibliography{references}
\end{document}